\documentclass[11pt]{article}

\usepackage{amsmath}
\usepackage{amssymb}
\usepackage{amscd}
\usepackage{amsthm}
\usepackage{indentfirst}

\usepackage[utf8]{inputenc} 

\newcommand{\N}{\ensuremath{\mathbb{N}}}
\newcommand{\Z}{\ensuremath{\mathbb{Z}}}
\newcommand{\Q}{\ensuremath{\mathbb{Q}}}
\newcommand{\C}{\ensuremath{\mathbb{C}}}
\newcommand{\F}{\ensuremath{\mathbb{F}}}

\usepackage[margin=2.6cm]{geometry}

\theoremstyle{plain}
\newtheorem{theorem}{Theorem}[section]
\newtheorem{lemma}[theorem]{Lemma}
\newtheorem{corollary}[theorem]{Corollary}

\newtheorem{conjecture}[theorem]{Conjecture}
\newtheorem*{assumption1}{Assumption (A1)}
\newtheorem*{assumption2}{Assumption (A2)}
\theoremstyle{definition}

\newtheorem{remark}[theorem]{Remark}
\newtheorem{example}[theorem]{Example}

\DeclareMathOperator{\Spec}{Spec}
\DeclareMathOperator{\Pic}{Pic}
\DeclareMathOperator{\Cl}{Cl}
\DeclareMathOperator{\rank}{rk}
\DeclareMathOperator{\Disc}{Disc}
\DeclareMathOperator{\divisor}{div}
\DeclareMathOperator{\Sel}{Sel}
\DeclareMathOperator{\PGL}{PGL}
\DeclareMathOperator{\GL}{GL}

\DeclareMathOperator{\ord}{ord}
\DeclareMathOperator{\orb}{orb}

\newcommand{\PP}{\mathbb{P}}
\newcommand{\A}{\mathbb{A}}

\begin{document}

\title{Galois covers of $\PP^1$ and number fields with large class groups\footnote{MSC classes: 11R29 (Primary) 11R16 (Secondary)}}

\author{Jean Gillibert \and Pierre Gillibert\footnote{The second author was supported by Austrian Science Foundation FWF, project P32337.}}

\date{September 2021}

\maketitle

\begin{abstract}
For each finite subgroup $G$ of $\PGL_2(\Q)$, and for each integer $n$ coprime to $6$, we construct explicitly infinitely many Galois extensions of $\Q$ with group $G$ and whose ideal class group has $n$-rank at least $\#G-1$. This gives new $n$-rank records for class groups of number fields.
\end{abstract}




\section{Introduction}


If $M$ is a finite abelian group, and if $m>1$ is an integer, we define the $m$-rank of $M$ to be the maximal integer $r$ such that $(\Z/m\Z)^r$ is a subgroup of $M$; we denote it by $\rank_m M$. If $K$ is a number field, we denote by $\Cl(K)$ the ideal class group of $K$.

Our motivation for the present paper is the following conjecture on class groups of number fields, which belongs to folklore, and is a consequence of the Cohen-Lenstra heuristics.

\begin{conjecture}
\label{conj1}
Let $d>1$ and $n>1$ be two integers. Then $\rank_n\Cl(K)$ is unbounded when $K$ runs through the number fields of degree $[K:\Q]=d$. 
\end{conjecture}

When $n=d$, and more generally when $n$ divides $d$, this conjecture follows easily from class field theory \cite{brumer65,RZ}.
On the other hand, when $n$ and $d$ are coprime, there is not a single case where Conjecture~\ref{conj1} is known to hold. For a survey of known results, see \cite{gl20}.

It was proved by Nakano \cite{Nakano84, Nakano85} that, given $n>1$ and $(r_1,r_2)\in \N^2$, there exist infinitely many number fields $K$ with $r_1$ real places and $r_2$ complex places such that
\begin{equation*}
\label{eqN1}
\rank_n \Cl(K)\geq r_2+1.
\end{equation*}

This is currently the best known result for general $(r_1,r_2)$ and $n$. To our knowledge, the only improvements to this general bound are for $(r_1,r_2)=(3,0)$ \cite{Nakano86}, or for specific values of $n$; in particular, a better bound was obtained by Levin \cite{levin07} in the case when $d\geq n^2$.

In the present paper, we improve on Nakano's inequality in the case when
$$
(r_1,r_2)= (4,0), (6,0), (0,3), (0,4) ~\text{and}~ (0,6)
$$
by constructing fields $K$ which are in addition Galois extensions of $\Q$. Our strategy is closely related to the ``geometric'' techniques developed in \cite{gl12} and \cite{bg18}, an overview of whose is given in \cite{gl20}. The main new ingredient of our work is the construction of specific Galois covers of $\PP^1_\Q$ from finite subgroups of $\PGL_2(\Q)$. 

Our main result is as follows:

\begin{theorem}
\label{thm:main}
Let $G\leq \PGL_2(\Q)$ be a finite subgroup. Let $n$ be an integer coprime to $6$. Then there exist infinitely many isomorphism classes of number fields $K$ such that
\begin{enumerate}
\item[(1)] $K/\Q$ is Galois with group $G$;
\item[(2)] $\rank_n \Cl(K) \geq \#G - 1$.
\end{enumerate}
\end{theorem}

The list of finite subgroups of $\PGL_2(\Q)$ is well known \cite[Prop.~1.1]{Beauville2010}:
\begin{equation}
\label{eq:grouplist}
\Z/2\Z,\quad \Z/3\Z,\quad \Z/4\Z,\quad \Z/6\Z,\quad D_2=(\Z/2\Z)^2,\quad D_3=\mathfrak{S}_3,\quad D_4,\quad D_6,
\end{equation}
where $D_r$ denotes the dihedral group of order $2r$.

The condition that $n$ is coprime to $6$ in Theorem~\ref{thm:main} can be relaxed depending on $G$. We sum up our results for each group $G$ in Table~\ref{table1}, with relevant references in the case of quadratic and cubic extensions, which were already known.

The proof of Theorem~\ref{thm:main} follows from a general strategy, and ends up with a case-by-case analysis. For each finite subgroup $G\leq \PGL_2(\Q)$, we give an explicit family of fields satisfying the conclusions of Theorem~\ref{thm:main}, and we count isomorphism classes of these fields, ordered by discriminant.

\begin{table}[ht]
\label{table1}
\caption{Table of $G$ and $n$ for which there exist infinitely many Galois extensions $K/\Q$ with group $G$ such that $\rank_n \Cl(K) \geq \#G - 1$. All these are current $n$-rank records, except for imaginary quadratic (resp. biquadratic) extensions, where the $n$-rank 2 (resp. 4) can be achieved, as was shown by Yamamoto \cite{Yamamoto70}.}
\bigskip
\centering
\begin{tabular}{|c|c|c|l|}
\hline
$G$ & $n$ & signature & author(s) \\
\hline
$\Z/2\Z$ & any & imaginary & Ankeny-Chowla 1955 \cite{AC55} \\
         & any & real & Yamamoto 1970 \cite{Yamamoto70} \\
         &     &      & Weinberger 1973 \cite{Weinberger73} \\
\hline
$\Z/3\Z$ & any & real & Nakano 1986 \cite{Nakano86} \\
\hline
$\Z/4\Z$ & odd & real & \\
\hline
$\Z/6\Z$ & coprime to $6$ & real & \\
\hline
$D_2=(\Z/2\Z)^2$ & any & imaginary & Yamamoto 1970 \cite{Yamamoto70} \\
                 & any & real & \\
\hline
$D_3=\mathfrak{S}_3$ & odd & imaginary & \\
\hline
$D_4$ & coprime to $6$ & imaginary & \\
\hline
$D_6$ & coprime to $6$ & imaginary & \\
\hline
\end{tabular}
\end{table}

\begin{example}
The splitting field of the polynomial
$$
x^6 + 3x^5 + 24829767x^4 + 49659529x^3 + 24829767x^2 + 3x + 1
$$
(obtained by specializing our polynomial $D_3P$ in \S{}\ref{sec:D3} at $y^n=19^7$) is a totally imaginary Galois extension of $\Q$ with group $\mathfrak{S}_3$. One computes with Pari/GP \cite{PARI2} that its ideal class group has $42$-rank exactly $5$.
\end{example}


\subsection*{Acknowledgements}

We thank Florence Gillibert and Gabriele Ranieri for inspiring email exchanges throughout the writing of this paper.
This work originated while the authors were visiting the Institut de Math{\'e}matiques de Vimont during the winter semester break 2019/2020.

Finally, this work would not have seen the light of day without the use of the Pari/GP software \cite{PARI2}, which allowed us to carry out various numerical experiments.


\section{Galois covers of $\PP^1$ whose Picard group has large $n$-rank}


Throughout this paper, by ``curve'' we mean a smooth projective geometrically irreducible curve defined over some field $k$.

Let $G$ be a finite group. A Galois cover of curves with group $G$ is a finite surjective morphism $\phi:C_1\to C_2$ of curves, together with an action of $G$ on $C_1$ under which $\phi$ is invariant, such that $C_1\to C_2$ is generically an \'etale $G$-torsor (i.e. $k(C_1)/k(C_2)$ is a Galois extension with group $G$).

\begin{lemma}
\label{lem:Galoiscover}
Let $k$ be a field. Let $G\leq \PGL_2(k)$ be a finite subgroup.
Let $a,b \in\PP^1(k)$ be two points which do not lie in the same orbit under the action of $G$. Then the map
$$
h:\PP^1\to \PP^1; \quad x \mapsto \prod_{\sigma\in G} \frac{x-\sigma(a)}{x-\sigma(b)}
$$
is a (geometrically connected) Galois cover with group $G$.
\end{lemma}

In the statement above, and in the rest of the paper, we fix an embedding $\A^1 \to \PP^1$ and identify rational functions on $\A^1$ and on $\PP^1$. In particular, given $x_0\in k$, the function $x-x_0$ has a unique zero at $x_0$ and a unique pole at $\infty$. We extend this property by declaring that the expression $x-\infty$ is the constant function equal to $1$. Thus, the rational function $h$ above is well-defined even if the orbit of $a$ or that of $b$ contains $\infty$.

If $G\leq \PGL_2(k)$ is a finite subgroup and $x_0\in \PP^1(\overline{k})$ is a point, we denote by $\orb(x_0)$ its orbit under the action of $G$, a slight abuse of notation.

\begin{proof}
We note that $h$ is a rational map whose divisor is
$$
\divisor(h)= \sum_{\sigma\in G} \sigma(a) - \sum_{\sigma\in G} \sigma(b)
$$

Now, for any $\tau\in G$, $h\circ\tau$ is a rational map, and
$$
\divisor(h\circ\tau) = \tau^{-1}(\divisor(h)) = \divisor(h)
$$
because  $\divisor(h)$ is invariant under the action of $\tau^{-1}$, which is an element of $G$. We know that there exists, up to a multiplicative constant, a unique rational map $\PP^1\to \PP^1$ with given divisor. Therefore, there exist $\mu_\tau\in k^\times$ such that $h\circ\tau=\mu_\tau h$. But we know that $\tau$ has at least one fixed point $x_0$ in $\PP^1(\overline{k})$. If this point $x_0$ is neither a zero nor a pole of $h$, then by evaluating $h\circ\tau$ at $x_0$ we see that $\mu_\tau=1$. If $x_0$ is a zero or a pole of $h$, then we let $r:=\ord_{x_0}(h)$ and the same conclusion holds by evaluating at $x_0$ the equality
$$
(x-x_0)^{-r}h(\tau(x)) = \mu_\tau (x-x_0)^{-r}h(x).
$$

This proves that $h:\PP^1\to \PP^1$ is invariant under the action of $G$. Moreover, $h$ is finite of degree $\#G$ by construction. Therefore, $h:\PP^1\to \PP^1$ induces an isomorphism $\PP^1\simeq \PP^1/G$ (quotient with respect to the fppf topology). Moreover, $G$ being finite, it acts freely on $\PP^1$ outside finitely many points, hence $h$ is generically a $G$-torsor. As $G$ is constant, $G$-torsors for the \'etale and the fppf topology coincide. We conclude that $h$ is a Galois cover with group $G$.
\end{proof}

\begin{remark}
The ramification points of $h$ are exactly the points of $\PP^1$ with non-trivial stabilizer under the action of $G$.
\end{remark}

\begin{remark}
Another way to state the conclusion of Lemma~\ref{lem:Galoiscover} is the following: the splitting field over $k(t)$ of the polynomial
\begin{equation}
\label{eq:Rt}
R_t := \prod_{\sigma\in G} (x-\sigma(a)) - t\prod_{\sigma\in G} (x-\sigma(b))
\end{equation}
is a regular Galois extension of $k(t)$ with group $G$.

While constructions of Galois extensions of $k(t)$ from finite subgroups of $\PGL_2(k)$ are classical \cite[Chap.~1]{Serretopics}, the way we achieve this from the orbits of two rational points does not seem to appear in the literature. It follows from \cite[\S{}1.1]{Serretopics} that, when $G \simeq \Z/3\Z$ (more generally, any group of odd order), or $G$ is the subgroup induced by the standard representation $\mathfrak{S}_3\to \GL_2(k)$, the fields obtained are generic, whatever the value of $(a,b)$ is. When $G=\Z/4\Z$ and $\sqrt{-1} \notin k$, there is no generic equation over $k(t)$ \cite[\S{}1.2]{Serretopics}, so our construction for different values of $(a,b)$ may lead to distinct families of fields.
\end{remark}

Combining Lemma~\ref{lem:Galoiscover} with Hilbert's irreducibility theorem yields the following:

\begin{corollary}
\label{cor:galoispol}
Assume $k$ is a number field, and let $G$, $a$ and $b$ satisfying the assumptions of Lemma~\ref{lem:Galoiscover}. Then for all $m\in k$ outside a thin set, the polynomial $R_m$ obtained by specialization of \eqref{eq:Rt} is irreducible, and its splitting field is a Galois extension of $k$ with group $G$. Moreover, the action of $G$ on its roots is induced by the natural action of $G$ on $\PP^1(\overline{\Q})$.
\end{corollary}

\begin{remark}
\label{rmk:signature}
A Galois extension of $\Q$ is either totally real or totally complex. Therefore, when $k=\Q$, the splitting field of $R_m$ is totally real if and only if $R_m$ has at least one real root.
\end{remark}

\begin{lemma}
\label{lem:Picard}
Let $G$, $a$ and $b$ satisfying the assumptions of Lemma~\ref{lem:Galoiscover}. Let $n>1$ be an integer which is coprime to the orders of the stabilizers of $a$ and $b$, and let $\lambda\in k^\times$.

Then the polynomial
$$
\prod_{\sigma\in G} (x-\sigma(a)) - \lambda y^n\prod_{\sigma\in G} (x-\sigma(b))
$$
defines a geometrically irreducible curve $C$ over $k$, such that:
\begin{enumerate}
\item[(1)] the $y$-coordinate map $C\to \PP^1$ is a Galois cover with group $G$;
\item[(2)] the Picard group of $C$ contains a subgroup isomorphic to $(\Z/n\Z)^{\# \orb(a) + \# \orb(b) -2}$.
\end{enumerate}
\end{lemma}

\begin{proof}
(1) Let $h=h_{a,b}$ be the map from Lemma~\ref{lem:Galoiscover}. Then $C$ is the fiber product
$$
\begin{CD}
C @>x>> \PP^1 \\
@VyVV @VV\lambda^{-1} h_{a,b}V \\
\PP^1 @> z\,\mapsto\, z^n>> \PP^1 \\
\end{CD}
$$

According to Lemma~\ref{lem:Galoiscover}, $\lambda^{-1}h_{a,b}$ is a Galois cover with group $G$. It follows that the same holds for the $y$-coordinate map $C\to \PP^1$, which is obtained by pulling-back $\lambda^{-1}h_{a,b}$.

(2) Considering $h_{a,b}$ as an element of $k(x)$, the equation of $C$ is equivalent to
\begin{equation}
\label{eq:totramif}
y^n = \lambda^{-1}h_{a,b}
\end{equation}

The linear factors at the numerator (resp. denominator) of $h_{a,b}$ appear with multiplicity equal to the order of the stabilizer of $a$ (resp. $b$), which by assumption are coprime to $n$. It follows that the map $x:C\to \PP^1$ has degree $n$ and is totally ramified above the set $\orb(a) \cup \orb(b)$. In particular, $C$ is geometrically irreducible.

Each of the points $c\in\orb(a) \cup \orb(b)$ has a unique preimage by the map $x$, which we denote by $T_c$. 
It follows that
$$
\divisor\left(\frac{x-c}{x-b}\right)=nT_c-nT_b \qquad \text{for all $c\in\orb(a) \cup \orb(b)$}.
$$

This yields $\# \orb(a) + \# \orb(b) -1$ nonzero divisors $T_c-T_b$ whose classes are $n$-torsion. On the other hand, it follows from \eqref{eq:totramif} that
\begin{equation}
\label{eq:nontrivialrelation}
\divisor(y)= \sum_{\sigma\in G} T_{\sigma(a)} - T_{\sigma(b)}
\end{equation}
and this is the only nontrivial relation between the divisor classes $T_c-T_b$. Therefore, these classes generate a subgroup of $\Pic(C)$ isomorphic to $(\Z/n\Z)^{\# \orb(a) + \# \orb(b) -2}$.
\end{proof}

\begin{remark}
\label{rmq:infty}
Under the assumptions of Lemma~\ref{lem:Picard}, if the orbit of $b$ contains $\infty$, then for all $c\in\orb(a) \cup \orb(b)$, we have $\divisor(x-c)=nT_c-nT_{\infty}$, where $T_c$ and $T_\infty$ are defined in the proof of Lemma~\ref{lem:Picard}.
\end{remark}

\begin{example}
\label{rmq:finitefields}
Let $q$ be a power of some prime number. Then for any $r\mid q-1$, the group $\mu_r(\F_q)$ acts on $\PP^1$ by multiplication, an action which stabilizes $0$ and $\infty$.
If in addition $r\neq q-1$, then there exists two points $a$ and $b$ in $\F_q^\times$ which have distinct orbits and trivial stabilizer, hence Lemma~\ref{lem:Picard} gives an explicit cyclic degree $r$ Galois cover $C\to \PP^1$ defined over $\F_q$ such that
$$
\rank_n \Pic(C) \geq 2r-2.
$$
\end{example}


\section{Arithmetic specialization}


If $K$ is a number field, we define
\begin{equation}
\label{Sn}
\Sel^n(K):=\{\gamma\in K^\times/(K^\times)^n; \text{$\forall v$ finite place of $K$, $v(\gamma)\equiv 0\pmod{n}$}\} 
\end{equation}
which is an analogue of the Selmer group for the multiplicative group over $K$. Then we have an exact sequence
\begin{equation*}
\begin{CD}
1 @>>> \mathcal{O}_{K}^{\times}/\big(\mathcal{O}_{K}^{\times}\big)^n @>>> \Sel^n(K) @>>> \Cl(K)[n] @>>> 0. \\
\end{CD}
\end{equation*}

In fact, $\Sel^n(K)$ is none other than the flat cohomology group $H^1_{\mathrm{fppf}}(\Spec(\mathcal{O}_K),\mu_n)$, and the exact sequence above is the Kummer exact sequence in flat cohomology.

According to \cite[Prop.~1.1]{gg19}, this exact sequence always splits. Therefore, we are able to deduce from \cite[Lemma~2.6]{gl12} that
\begin{equation}
\label{eq:SelRank}
\rank_n \Sel^n(K) = \rank_n \mathcal{O}_{K}^{\times}/\big(\mathcal{O}_{K}^{\times}\big)^n + \rank_n \Cl(K)[n].
\end{equation}

The following statement is a variant of \cite[Theorem~2.4]{gl12}. It is obtained by combining a quantitative version of Hilbert's irreducibility theorem, due to Cohen \cite{cohen1981}, with the equality \eqref{eq:SelRank} above. See also \cite[Theorem~2.7]{gl20}.

\begin{theorem}
\label{thm:specialization}
Let $C$ be a smooth projective geometrically irreducible curve over $\Q$, and let $n>1$ be an integer. Let $D_1,\dots, D_s$ be divisors on $C$ whose classes in $\Pic(C)$ generate a subgroup isomorphic to $(\Z/n\Z)^s$, and let $g_1,\dots,g_s$ be rational functions on $C$ such that $\divisor(g_i)=nD_i$ for all $i$. Assume that there exists a finite map $\phi:C\to\PP^1$ of degree $d>1$ such that, for all $t\in\N$, the point $P_t:=\phi^{-1}(t)$ has the property that
\begin{equation}
\label{eq:SelCond}
\tag{SC}
g_1(P_t),\dots,g_s(P_t) ~\text{define classes in}~ \Sel^n(\Q(P_t)),
\end{equation}
where $\Sel^n$ is defined in \eqref{Sn}. Then for all but $O(\sqrt{N})$ values $t\in\{1,\dots, N\}$, the field $\Q(P_t)$ satisfies $[\Q(P_t):\Q]=d$ and
$$
\rank_n \Cl(\Q(P_t))\geq s-\rank_\Z \mathcal{O}_{\Q(P_t)}^\times.
$$
Moreover, there are infinitely many isomorphism classes of such fields $\Q(P_t)$.
\end{theorem}

Let us mention quantitative versions of the last statement: it was proved by 
Dvornicich and Zannier \cite{DZ94} that, for $N$ large enough, the number of isomorphism classes in the set
$$
\{ \Q(P_1),\dots, \Q(P_N) \}
$$
is $\gg N/\log N$, where the implicit constant depends on $\phi$ and $C$.
Moreover, if the map $\phi$ has at least three distinct zeroes in $C(\overline{\Q})$, then by a result of Corvaja and Zannier \cite[Corollary~1]{cz03}, for $N$ large enough, the number of isomorphism classes is $\gg N$ (here, we use Vinogradov's notation: $g\gg f$ means that $f=O(g)$ when $x\to +\infty$). For more details, we refer the reader to the introduction of \cite{bl17}.


\subsection{Strategy of proof of Theorem~\ref{thm:main}}

Our case-by-case strategy for proving Theorem~\ref{thm:main} is as follows: for each group $G$ in the list \eqref{eq:grouplist}, and for each suitable integer $n$, we shall explicitly give
\begin{enumerate}
\item one (or two) homographies which generate a subgroup isomorphic to $G$ in $\PGL_2(\Q)$;
\item two points $a,b \in\PP^1(\Q)$ satisfying the assumptions of Lemma~\ref{lem:Picard}, and such that the orbit of $b$ contains $\infty$;
\item a corresponding polynomial defining a curve $C$ as in Lemma~\ref{lem:Picard} (\emph{i.e.} a choice of $\lambda$);
\item a set of rational functions $g_1,\dots,g_s$ ($s={\# \orb(a) + \# \orb(b) -2}$) on $C$ with $\divisor(g_i)=nD_i$, such that the divisors $D_i$ generate a subgroup of $\Pic(C)$ isomorphic to $(\Z/n\Z)^s$.
\item a congruence condition on $y\in \Z$ such that the corresponding point $P_y$ on the curve $C$ satisfies the condition \eqref{eq:SelCond}.
\end{enumerate}

Then one can construct from the congruence condition $y\equiv y_0 \pmod{N}$ a map $\phi=\frac{y-y_0}{N}$ which satisfies the hypotheses of Theorem~\ref{thm:specialization}.
By Lemma~\ref{lem:Picard}, the map $y:C\to \PP^1$ is a Galois cover with group $G$, hence the same holds for the map $\phi$.

It then follows from Theorem~\ref{thm:specialization} that, for ``most'' values of $y$ satisfying the congruence condition, the field $\Q(P_y)$ is a Galois extension of $\Q$ with group $G$, and satisfies
$$
\rank_n \Cl(\Q(P_y))\geq s-\rank_\Z \mathcal{O}_{\Q(P_y)}^\times.
$$

Finally, it suffices to determine the signature of the field $\Q(P_y)$ in order to deduce, by Dirichlet's unit theorem, an explicit lower bound on its $n$-rank.


\subsection{Detailed outline of the proof}
\label{detailedoutline}

Let us explain in more detail how each of the items in the general strategy is achieved in practice. We do not include the cases of $\Z/2\Z$ and $D_2=(\Z/2\Z)^2$ in this general discussion, since these are slightly different. We shall treat them in the case-by-case analysis instead.

\subsubsection{Choice of the homographies}
\label{3.2.1}

We make, once and for all, the same choice as \cite[Prop.~1.1]{Beauville2010}. Let us fix $r\geq 3$, and let $\zeta$ be a primitive $r$-th root of $1$, then the homography
\begin{equation}
\label{eq:homography}
f:z\mapsto \frac{(\zeta +\bar\zeta+1)z-1}{z+1}
\end{equation}
is an element of order $r$ in $\PGL_2(\C)$. Together with $z\mapsto 1/z$, it generates a subgroup of $\PGL_2(\C)$ isomorphic to $D_r$. It is clear that, for $r=3,4,6$, these homographies are defined over $\Q$.

We note that $f$ satisfies $f(0)=-1$ and $f(-1)=\infty$.

\begin{lemma}
\label{lem:IDTS}
Let $r\geq 3$, and let $G=\langle f\rangle\simeq \Z/r\Z$ or $G=\langle f,z\mapsto 1/z\rangle \simeq D_r$, where $f \in \PGL_2(\C)$ is defined above. Then for all $b\in \orb(\infty)\cap \C$ the rational map
$$
\varphi : z\mapsto \prod_{\sigma \in G} (\sigma(z)-b)
$$
is constant. In particular,
\begin{equation}
\label{eq:IDTS}
\prod_{\sigma \in G} \sigma(z) = 1 \quad \text{and}\quad \prod_{\sigma \in G} (\sigma(z)+1) = (\zeta+1)^{\#G}.
\end{equation}
\end{lemma}

\begin{proof}
In fact, the first statement holds for any finite subgroup $G$ of $\PGL_2(\C)$. Indeed, let $b\in \orb(\infty)\cap \C$ and let $\varphi$ be the rational map defined above; then the divisor of $\varphi$ is given by
$$
\divisor(\varphi)= \sum_{\sigma \in G} (\sigma - b)^{-1}(0) - \sum_{\sigma \in G} (\sigma - b)^{-1}(\infty)
= \sum_{\sigma \in G} \sigma^{-1}(b) - \sum_{\sigma \in G} \sigma^{-1}(\infty)=0
$$
where the last equality holds because $b$ and $\infty$ have the same orbit. Thus, $\varphi$ is a constant map. If we let $b=0$ and $b=-1$ respectively, which belong to the orbit of $\infty$, we obtain that the two quantities in \eqref{eq:IDTS} are constants. In the case when $G= \langle f\rangle$, these constants can be computed by observing that $f(\zeta)=\zeta$, hence $\zeta$ is a fixed point of $G$. In the case when $G=\langle f,z\mapsto 1/z\rangle$, we observe that
$$
G= \{f^s, 1/f^s ~|~ s=1,\dots, r\}
$$
hence first equality in \eqref{eq:IDTS} follows immediately from the previous case. The second one can be worked out explicitly:
$$
\prod_{\sigma \in G} (\sigma(\zeta)+1) = (\zeta+1)^r(\zeta^{-1}+1)^r = (\zeta+1)^r(\zeta+1)^r\zeta^{-r}=(\zeta+1)^{2r},
$$
hence the result.
\end{proof}

\begin{remark}
\label{rmq:fzeroes}
A non-constant homography can be determined from its zero, its pole, and its value at $1$. By applying this to the powers of our homography $f$, which satisfies $f^2(0)=\infty$, we deduce that for any integer $s$ there exists a constant $\mu_s\neq 0$ such that
$$
f^s(z) = \mu_s \cdot\frac{z-f^{-s}(0)}{z-f^{2-s}(0)}.
$$
\end{remark}

\begin{remark}
\label{rmq:GoodRed}
We note that \eqref{eq:homography} can be defined by a matrix of determinant $\zeta+\bar\zeta +2$. Therefore, for any integer $m$ coprime to
$$
\zeta+\bar\zeta +2
= \left\{
\begin{array}{ll}
1 & \text{if $r=3$} \\
2 & \text{if $r=4$} \\
3 & \text{if $r=6$}
\end{array}\right.
$$
the homography \eqref{eq:homography} reduces into a non-constant homography of the projective line over $\Z/m\Z$. In other terms, it has good reduction outside $\zeta+\bar\zeta +2$.
\end{remark}

\subsubsection{Choice of the points $a$ and $b$}
\label{3.2.2}

We put $b=0$, whose orbit under $G$ contains infinity, as noted above. We denote by $\omega$ the order of the stabilizer of $0$ under the action of $G$, namely:
$$
\omega
= \left\{
\begin{array}{ll}
1 & \text{if $G$ is cyclic} \\
2 & \text{if $G$ is dihedral.}
\end{array}\right.
$$

In order to maximize the Picard group of $C$, we choose $a$ such that $\#\orb(a)=\#G$, in other terms its stabilizer under the action of $G$ is trivial. Apart from that, the choice of $a$ is far from being canonical, although we aim to choose the ``smallest'' integer which is not in the orbit of $0$. This choice has an impact on the rest of the process, in particular the congruence conditions on $y^n$ (hence the conditions on $n$) depend on it.

\subsubsection{Choice of the curve}
\label{3.2.3}

We note that the polynomial $\prod_{\sigma\in G} (x-\sigma(a))$ is monic of degree $\#G$, with constant coefficient $(-1)^{\#G}$ by \eqref{eq:IDTS}, and the polynomial $\prod_{\sigma\in G} (x-\sigma(0))$ has degree $\#G-\omega$, and has zero constant coefficient. Therefore, given $\lambda\in \Q^\times$, the polynomial (in the variable $x$)
\begin{equation}
\label{eq:mypoly}
\prod_{\sigma\in G} (x-\sigma(a)) - \lambda y^n\prod_{\sigma\in G} (x-\sigma(0))
\end{equation}
is monic, with constant coefficient $(-1)^{\#G}$.

We choose $\lambda$ in order to rescale the polynomials $\prod_{\sigma\in G} (x-\sigma(a))$ and $\prod_{\sigma\in G} (x-\sigma(0))$ into polynomials with integral coefficients, whose gcd is $1$, and whose leading coefficient is positive (in other terms, polynomials with content $1$). So, given $a$, our choice of a curve $C$ as in Lemma~\ref{lem:Picard} is canonical.

More precisely, let us write as irreducible fractions $\beta_i/\alpha_i$ the elements of $\orb(a)$, which by construction are rational numbers (we choose here $\alpha_i>0$, hence the fraction is unique). Similarly we denote by $\delta_j/\gamma_j$ the elements of $\orb(0)\setminus\{\infty\}$. We note that $\prod_i \alpha_i = \prod_i \beta_i$ according to \eqref{eq:IDTS}. Then we put
$$
\lambda:= -\left(\prod_j \gamma_j\right)^\omega\times\left(\prod_i \alpha_i\right)^{-1}.
$$

Given our choice of $\lambda$, when multiplying \eqref{eq:mypoly} by $\prod_i \alpha_i$, we obtain the polynomial
\begin{equation}
\label{eq:mypoly2}
\prod_{i=1}^{\#\orb(a)} (\alpha_i x-\beta_i) + y^n\prod_{j=1}^{\#\orb(0)-1} (\gamma_j x -\delta_j)^\omega
\end{equation}

We note that the second product has $\#\orb(0)-1$ factors because the point at infinity corresponds (by convention) to the factor $1$, hence has been removed from the list.

If $y\in \Z$ is chosen in such a way that the polynomial \eqref{eq:mypoly} has integral coefficients (equivalently, the polynomial \eqref{eq:mypoly2} is zero modulo $\prod_i \alpha_i$), then its roots are algebraic units. This is our first assumption.

\begin{assumption1}
The integer $y\in \Z$ has the property that the polynomial \eqref{eq:mypoly} has integral coefficients.
\end{assumption1}

\begin{lemma}
\label{lem:CongCond}
If $G$ is cyclic, the condition
$$
y^n \equiv \lambda^{-1}\sum_{\sigma\in G} \sigma(a) \pmod{\prod \alpha_i}
$$
implies assumption~(A1). Likewise, if $G$ is dihedral, the condition
$$
y^n \equiv -\lambda^{-1}\sum_{\substack{(\sigma,\tau)\in G^2 \\ \sigma\neq \tau}} \sigma(a)\tau(a) \pmod{\prod \alpha_i}
$$
implies assumption~(A1).
\end{lemma}

\begin{proof}
Assume $G=\langle f\rangle$ is cyclic. By construction, the equation of $C$ can be written as
$$
y^n=-\frac{\prod_{i=1}^{\#\orb(a)}(\alpha_i x-\beta_i)}{\prod_{j=1}^{\#\orb(0)-1}(\gamma_j x -\delta_j)}.
$$

Let $\prod \alpha_i=\prod \beta_i=(\zeta+\bar\zeta +2)^vm$, where $m$ is coprime to $(\zeta+\bar\zeta +2)$, then by Remark~\ref{rmq:GoodRed} the reduction of $f$ modulo $m$ is an automorphism of the projective line over $\Z/m\Z$, hence by reducing the equation above modulo $m$ we obtain on the right-hand side a rational map whose zeroes and poles are $G$-invariant. But its numerator vanishes at zero, since its constant term is $\pm\prod \beta_i$. It follows that this rational map is constant mod $m$, which is equal to $\prod \alpha_i$ when $r=3$.

On the other hand, given a homography $\varphi$ and an irreducible fraction $\beta/\alpha$, an elementary computation shows that $\varphi(\beta/\alpha)\equiv\varphi(0) \pmod{\beta}$, including the case when $\varphi(0)=\infty$. It follows immediately from this observation that the orbit of $a$ and that of $0$ coincide after reduction modulo any of the $\beta_i$. When $r=4$ we note that there is exactly one element of the orbit of $a$ which reduces to $0$ modulo $\beta_i$, because the other elements in the orbit are $\infty$ and $\pm 1$. It follows that the $\beta_i$ are coprime, hence the Chinese remainder Theorem implies that our rational map is constant modulo $\prod \beta_i$. When $r=6$, the orbit of $0$ is $\{0,\infty, \pm 1, 2, 1/2\}$, so there are two points which reduce to zero modulo $2=\zeta+\bar\zeta+1$, and $2$ is the only prime for which this happens. But $2$ is coprime to $3=(\zeta+\bar\zeta +2)$, hence $2\mid m$ and the first part of our proof implies the result.

We have just proved that our map is constant modulo $\prod \alpha_i$, hence assumption~(A1) is satisfied whenever $y^n$ is equal to this constant modulo $\prod \alpha_i$. By construction, we have
$$
-\frac{\prod_{i=1}^{\#\orb(a)}(\alpha_i x-\beta_i)}{\prod_{j=1}^{\#\orb(0)-1}(\gamma_j x -\delta_j)} = -\lambda^{-1} \left(\frac{x^r-(\sum_{\sigma\in G}\sigma(a)) x^{r-1}+\cdots}{x^{r-1}+\cdots}\right)
$$
hence the constant is $\lambda^{-1}\sum_{\sigma\in G} \sigma(a)$.

The dihedral case follows by a similar argument, with an additional grain of salt: the denominator of the rational map, corresponding to the orbit of $0$, is the square of the previous one. Therefore, the condition obtained is on the coefficient of degree $2r-2$ of the numerator of the rational map.
\end{proof}

\begin{remark}
Our proof includes a case-by-case analysis  for the reader's convenience, but it clearly holds for general values of $r$, replacing $\Z$ by the ring of integers of $\Q(\zeta+\bar\zeta)$.
\end{remark}

\subsubsection{Choice of the functions}
\label{3.2.4}

According to Remark~\ref{rmq:infty}, the functions $x-c$, where $c$ runs through $\orb(a) \cup \orb(0)$, are natural building blocks for the maps $g_i$ in the sense that any product or quotient of such functions has a divisor which is a multiple of $n$, and the only nontrivial relation between the corresponding divisor classes comes from \eqref{eq:nontrivialrelation}.

Since $b=0$, the function $x$ is a natural candidate. In fact, under assumption~(A1) from \S{}\ref{3.2.3}, the value of $x$ at $P_y$ is an algebraic unit, hence defines a class in the Selmer group. Clearly, the same holds for all other Galois conjugates of $x$, which are:
\begin{align*}
f(x) &= \frac{(\zeta + \bar\zeta + 1)x - 1}{x+1} \\
f^2(x) &= \frac{(\zeta + \bar\zeta)x - 1}{x} \\
f^3(x) &= \frac{(\zeta^2 + {\bar\zeta}^2 + \zeta + \bar\zeta + 1)x - (\zeta + \bar\zeta + 1)}{(\zeta + \bar\zeta + 1)x - 1} \\
f^4(x) &= \frac{(\zeta^2 + {\bar\zeta}^2 + 1)x - (\zeta + \bar\zeta)}{(\zeta + \bar\zeta)x - 1} \\
f^5(x) &= \frac{(\zeta^3 + {\bar\zeta}^3 + \zeta^2 + {\bar\zeta}^2 + \zeta + \bar\zeta + 1)x - (\zeta^2 + {\bar\zeta}^2 + \zeta + \bar\zeta + 1)}{(\zeta^2 + {\bar\zeta}^2 + \zeta + \bar\zeta + 1)x - (\zeta + \bar\zeta + 1)}
\end{align*}

For $r\leq 6$, this list is complete (and it is redundant if $r\leq 5$). We note that the numerators and denominators of these rational maps are of the form $\gamma_j x - \delta_j$, where $\delta_j/\gamma_j$ runs through $\orb(0)$ (notation from \S{}\ref{3.2.3}), as it was pointed in Remark~\ref{rmq:fzeroes}.

It follows from the computation above that, when $\delta_j/\gamma_j$ runs through $\orb(0)$, then either $\gamma_j x -\delta_j$ or $(\gamma_j x -\delta_j)/(x+1)$ is amongst the functions
$$
x, \quad f(x), \quad xf^2(x), \quad f(x)f^3(x), \quad xf^2(x)f^4(x), \quad f(x)f^3(x)f^5(x).
$$

Under assumption~(A1), the values of these functions at $P_y$ are units. These functions will be the first elements in our list, which we denote by $\varphi_j$ for $j=1,\dots,\#\orb(0)-2$. When $r=3$, we let $\varphi_1=x$ and $\varphi_2=x+1$, which, according to the following lemma, both specialize into units under assumption~(A1).

\begin{lemma}
Let $x:=x(P_y)$, then under assumption~(A1) the algebraic integer $x+1$ divides $\zeta+\bar\zeta +2$.
\end{lemma}

\begin{proof}
Indeed, $f(x)$ being a unit, $x+1$ divides $(\zeta + \bar\zeta + 1)x - 1$, and it follows that $x+1$ divides $(\zeta + \bar\zeta + 1)(x+1)-(\zeta + \bar\zeta + 1)x + 1 = \zeta+\bar\zeta +2$.
\end{proof}

\begin{lemma}
\label{lem:unsurdeux}
Assume $r=4,6$. Then for all $c\in \Q$, exactly half of the elements of $\orb(c)$ reduce to $-1$ modulo $\zeta+\bar\zeta +2$.
\end{lemma}

\begin{proof}
Under these assumptions, $\zeta+\bar\zeta +2$ is a prime number. Since the map $z\mapsto 1/z$ is well-defined modulo $\zeta+\bar\zeta +2$ and fixes $-1$, it is enough to consider the cyclic case. We first observe that
$$
f(z)=\frac{(\zeta+\bar\zeta +2)z}{z+1}-1.
$$
Let $\beta/\alpha$ be an element of $\orb(c)$, where $\alpha$ and $\beta$ are coprime integers. Then $\beta/\alpha$ reduces to $-1$ modulo $\zeta+\bar\zeta +2$ if and only if $\zeta+\bar\zeta +2$ divides $\alpha+\beta$. Now, according to the previous formula, we have
$$
f(\alpha/\beta)=\frac{(\zeta+\bar\zeta +2)\beta}{\alpha+\beta}-1.
$$
Assume that $\beta/\alpha$ does not reduce to $-1$. Then the denominator $\alpha+\beta$ does not reduce to zero, hence $f(\alpha/\beta)$ reduces to $-1$. On the other hand, if $\beta/\alpha$ reduces to $-1$, then the denominator $\alpha+\beta$ reduces to zero, and $\beta$ does not (because $\beta\neq 0$ and $\gcd(\alpha,\beta)=1$). So, the quantity $\frac{(\zeta+\bar\zeta +2)\beta}{\alpha+\beta}$ does not reduce to zero, hence $f(\alpha/\beta)$ does not reduce to $-1$. This proves that exactly one out of two elements in the orbit of $c$ reduces to $-1$.
\end{proof}

We shall now give the key ingredient of our proof: 
the equation of our curve $C$ being equivalent to the vanishing of \eqref{eq:mypoly2}, the following relation holds in the function field of $C$
\begin{equation}
\label{eq:curveSDN}
\prod_{i=1}^{\#\orb(a)} (\alpha_i x-\beta_i) = - y^n\prod_{j=1}^{\#\orb(0)-1} (\gamma_j x -\delta_j)^\omega
\end{equation}
where the $\beta_i/\alpha_i$ run through $\orb(a)$ (notation from \S{}\ref{3.2.3}). We observe that $(x+1)$ is amongst the factors on the right-hand side, since $-1$ belongs to the orbit of $0$.

If $r=3$ we keep this relation unchanged. If $r=4,6$, we let
$$
\varepsilon_i :=  \left\{
\begin{array}{ll}
1 & \text{if $\zeta+\bar\zeta +2$ divides $\alpha_i+\beta_i$} \\
0 & \text{otherwise}
\end{array}\right.
$$

We claim that exactly half of the $\varepsilon_i$ are equal to one. Indeed, $\zeta+\bar\zeta +2$ divides $\alpha_i+\beta_i$ if and only if $\beta_i/\alpha_i$ reduces to $-1$ modulo $\zeta+\bar\zeta +2$, and, according to Lemma~\ref{lem:unsurdeux}, exactly half of the elements of the orbit of $a$ have this property. 

Dividing both sides of \eqref{eq:curveSDN} by $(x+1)^{\#G/2}$, we obtain the equality
\begin{equation}
\label{eq:curveDN}
\prod_{i=1}^{\#\orb(a)} \frac{(\alpha_i x-\beta_i)}{(x+1)^{\varepsilon_i}} = - y^n\prod_{j=1}^{\#\orb(0)-2} \varphi_j^\omega
\end{equation}

We denote by $\psi_i$ the functions
$$
\psi_i:=\frac{(\alpha_i x-\beta_i)}{(x+1)^{\varepsilon_i}} \qquad \text{for}~ i=1,\dots,\#\orb(a).
$$

Then the values of these functions at $P_y$ are algebraic integers: if $\varepsilon_i=0$ this is trivial; if $\varepsilon_i=1$ one can write
$$
\frac{(\alpha_i x-\beta_i)}{(x+1)} = \alpha_i - \frac{(\alpha_i + \beta_i)}{(x+1)}
$$
and the right-hand side is an algebraic integer because $x+1$ divides  $\zeta+\bar\zeta +2$ which itself divides  $\alpha_i+\beta_i$ under our assumptions.

The functions $\psi_i$ being defined for $r=4,6$, the equation \eqref{eq:curveDN} becomes
\begin{equation}
\label{eq:curveFN}
\prod_{i=1}^{\#\orb(a)} \psi_i = - y^n\prod_{j=1}^{\#\orb(0)-2} \varphi_j^\omega
\end{equation}

By construction, the $\varphi_j(P_y)$ are units and the $\psi_i(P_y)$ are algebraic integers. We shall prove that, under some additional condition on $y$, the $\psi_i(P_y)$ are coprime. If this holds then, by unique factorization into prime ideals in the ring of integers of $\Q(P_y)$, since the valuations of the right-hand side are multiples of $n$, the same holds for each of the factors $\psi_i(P_y)$ of the left-hand side, in other terms, the Selmer property is satisfied by the $\psi_i$.

If $r=3$, we let $\psi_i=\alpha_i x-\beta_i$ and consider the equation \eqref{eq:curveSDN} instead, which reads
\begin{equation}
\label{eq:curveFN3}
\prod_{i=1}^{\#\orb(a)} \psi_i = - y^n \varphi_1^\omega\varphi_2^\omega
\end{equation}
The same reasoning as above applies to this situation as well.

\begin{lemma}
\label{lem:pgcd}
Let $(\alpha,\beta)$ and $(\gamma,\delta)$ be two distinct pairs of coprime integers in $\N^*\times \Z$. Then for any algebraic integer $x$, the ideal $(\alpha x -\beta, \gamma x -\delta)$ divides
$$
\frac{\alpha\delta-\beta\gamma}{\gcd(\alpha,\gamma)}.
$$
\end{lemma}

\begin{proof}
The identity
$$
\frac{\gamma}{\gcd(\alpha,\gamma)}(\alpha x -\beta) - \frac{\alpha}{\gcd(\alpha,\gamma)}(\gamma x -\delta) = \frac{\alpha\delta-\beta\gamma}{\gcd(\alpha,\gamma)}
$$
proves that the desired quantity belongs to the ideal $(\alpha x -\beta, \gamma x -\delta)$, hence the result.
\end{proof}

\begin{remark}
The above Lemma is generically optimal in the sense that, under its assumptions, the ideal $(\alpha X -\beta, \gamma X -\delta)$ of the polynomial ring $\Z[X]$ satisfies
$$
(\alpha X -\beta, \gamma X -\delta) \cap \Z = \left(\frac{\alpha\delta-\beta\gamma}{\gcd(\alpha,\gamma)}\right).
$$
\end{remark}

So, we require $y$ to be coprime to all the quantities obtained from the above lemma:

\begin{assumption2}
The integer $y$ is coprime to
\begin{equation}
\label{eq:Ycoprimeto}
\prod_{i\neq j}\frac{\alpha_i\beta_j-\beta_i\alpha_j}{\gcd(\alpha_i,\alpha_j)}
\end{equation}
where $\beta_i/\alpha_i$ run through $\orb(a)$.
\end{assumption2}

At the end of the day, we obtain a list of rational functions $\varphi_j$, $\psi_i$, of the form
$$
\alpha x - \beta\quad \text{or}\quad \frac{\alpha x - \beta}{x+1}
$$
where $\alpha, \beta$ are coprime integers such that $\beta/\alpha$ belongs to $\orb(a) \cup \orb(0)$. If $r=4,6$, the number of these functions is $\#\orb(a)+\#\orb(0)-2$, and if $r=3$, the number of functions is $\#\orb(a)+2=\#\orb(a)+\#\orb(0)-1$.

In both cases, these functions are not multiplicatively independent, because they satisfy the non-trivial relation \eqref{eq:curveFN} (resp. \eqref{eq:curveFN3}), so in order to find an independent set we remove one of these functions from the list. If $r=3$, we obtain $\#\orb(a)+\#\orb(0)-2$ independent functions, which is the required number (it coincides with the $n$-rank of the Picard group). If $r=4,6$, in order to complete the list we add a last function
\begin{equation}
\label{eq:lastfunction}
\frac{(x+1)^2}{\zeta+\bar\zeta+2}
\end{equation}
which is multiplicatively independent from the previous ones.

\begin{lemma}
For $r=4,6$, the value of \eqref{eq:lastfunction} at $P_y$ in an algebraic number of norm $\pm 1$.
\end{lemma}

\begin{proof}
Let $Q$ be the polynomial \eqref{eq:mypoly}. Then the minimal polynomial of $x+1$ is $Q(T-1)$, and its constant coefficient $Q(-1)=(-1)^{\#G}\prod_{\sigma\in G} (\sigma(a)+1)$ is equal (up to sing) to the norm of $x+1$. According to \eqref{eq:IDTS}, this quantity is $(\zeta+1)^{\#G}$. Assume $G$ is cyclic, then we have
$$
(\zeta+1)^r = ((\zeta+1)^2)^{r/2} = \zeta^{r/2}(\zeta+\bar\zeta+2)^{r/2} =-(\zeta+\bar\zeta+2)^{r/2},
$$
from which one deduces that $(x+1)^2/(\zeta+\bar\zeta+2)$ has norm $\pm 1$, since $Q$ has degree $r$. The dihedral case can be proved along the same lines.
\end{proof}

Under assumption~(A1) the value of this last function \eqref{eq:lastfunction} at $P_y$ need not be a unit, although its minimal polynomial has coefficients in $\Z[(\zeta+\bar\zeta+2)^{-1}]$. We shall numerically check that, under an additional congruence condition on $y^n$, this polynomial has integral coefficients, so the value of \eqref{eq:lastfunction} at $P_y$ is a unit, hence defines a class in the Selmer group.

\subsubsection{Determination of the congruence condition}
\label{3.2.5}

In order to complete our proof, we shall explicitly determine, in each case, conditions on the integer $y$ of the form:
\begin{enumerate}
\item[(C1)] $y^n\equiv t_0\pmod{m}$, where $m=(\zeta+\bar\zeta+2)^v\prod_i \alpha_i$ for some $v\geq 0$.
\item[(C2)] $y$ is coprime to $p_1\cdots p_l$, where the $p_i$ are the prime factors of \eqref{eq:Ycoprimeto} which do not divide $m$.
\end{enumerate}

These conditions are chosen in such a way that (C1) and (C2) imply both (A1) and (A2), and additionally that the minimal polynomial of \eqref{eq:lastfunction} has integral coefficients.

In practice, the integers $m$ and $t_0$ in condition~(C1) are determined by a case-by-case computation, with the help of Pari/GP: we start from the congruence condition modulo $\prod_i \alpha_i$ given by Lemma~\ref{lem:CongCond}, which implies (A1). Then we compute the minimal polynomial of \eqref{eq:lastfunction} and require it to have integral coefficients. This gives us a slightly stronger condition modulo $(\zeta+\bar\zeta+2)^v\prod_i \alpha_i$ for some $v\geq 0$.

Next, we observe that $t_0$ is coprime to $m$ and to $p_1\cdots p_l$, in particular (C1) and (C2) imply that $y^n$ is coprime to \eqref{eq:Ycoprimeto}, so that assumption~(A2) is satisfied.

It follows from Bézout's identity that the equation $y^n\equiv t_0\pmod{m}$ has a solution if and only if the order of $t_0$ in $(\Z/m\Z)^\times$ is coprime to $n$. So, a first step is to assume that $n$ satisfies this condition. In all our examples, we tried to find the weakest possible condition on $n$, like $n$ is coprime to $2$, to $3$, or to $6$ depending on the cases. There does not seem to exist a direct relation between the order of $G$ and the condition on $n$.

Once the condition on $n$ is found, we describe explicitly the solutions of the equation $y^n\equiv t_0\pmod{m}$, which yields congruence conditions of the form $y\equiv y_0 \pmod{m}$, where $y_0$ is a suitable power of $t_0$, depending on the class of $n$ modulo the multiplicative order of $t_0$ in $(\Z/m\Z)^\times$.

Finally, since $t_0$ is coprime to $m$ and to $p_1\cdots p_l$, so does $y_0$, and the condition
$$
y\equiv y_0 \pmod{mp_1\cdots p_l}
$$
implies that both (C1) and (C2) above hold (it is in fact much stronger).
As explained above, this yields a map $\phi$ satisfying the hypotheses of Theorem~\ref{thm:specialization}.  We shall not exhibit such a map, but its existence allows us to make use of Theorem~\ref{thm:specialization}.

\begin{remark}
Although we believe it should exist, we have not been able to prove the existence of conditions similar to (C1) and (C2) in full generality.
\end{remark}

\subsubsection{Counting the fields obtained}
\label{3.2.6}

Given a curve as in Lemma~\ref{lem:Picard}, the zeroes of the map $y$ (hence those of the map $\phi$ whose existence is granted by the discussion at the end of \S{}\ref{3.2.5}) are in bijection with $\orb(a)$. Therefore, if $\#\orb(a)\geq 3$, then by the quantitative version of Theorem~\ref{thm:specialization} the number of isomorphism classes of fields in $\{ \Q(P_1),\dots, \Q(P_N) \}$ is $\gg N$. Since $\#\orb(a)=\#G$, this is the case for all our groups, except $\Z/2\Z$, which for simplicity we exclude from this discussion.

Finally, we note that the discriminant of our polynomial \eqref{eq:mypoly} is $O(y^{ln})$ for some integer $l$.
 Needless to say, the same holds for the discriminant of the field $\Q(P_y)$. It follows that, for sufficiently large positive $X$, the number of isomorphism classes of fields $\Q(P_y)$ satisfying $\rank_n \Cl(\Q(P_y)) \geq \#G - 1$ and such that $|\Disc(\Q(P_y)/\Q)|\leq X$ is $\gg X^{1/ln}$.


\section{Case-by-case proof of Theorem~\ref{thm:main}}



\subsection{$\Z/2\Z$}

We shall recover the result of Yamamoto \cite{Yamamoto70} and Weinberger \cite{Weinberger73} on real quadratic fields with $n$-rank at least one. We note that our family is not the same as theirs.
         
The homography $z\mapsto 1/z$ has order $2$; choosing $a=2$, $b=0$ and $\lambda=-1/2$ as in \S{}\ref{3.2.3} we obtain from Lemma~\ref{lem:Picard} the polynomial
\begin{equation*}
\label{eq:C2pol}
\begin{split}
C_2P &= \frac{1}{2}\left((x-2)(2x-1)+y^nx\right) \\
     &= x^2 + \left(\frac{y^n - 5}{2}\right)x + 1.
\end{split}
\end{equation*}

Let $y\in \Z$, and let $P_y$ be the corresponding point on the curve defined by $C_2P$.

\subsubsection*{Congruence conditions}
\begin{enumerate}
\item[(i)] $2\nmid y$ and $3\nmid y$.
\end{enumerate}

We claim that, under these conditions, the values at $P_y$ of the multiplicatively independent rational functions
\begin{equation*}
\label{eq:SelmerlistC2}
x, \quad x-2,
\end{equation*}
define classes in $\Sel^n(\Q(P_y))$. In other terms, the condition \eqref{eq:SelCond} in Theorem~\ref{thm:specialization} is satisfied by these rational maps and this choice of $P_y$.

Abusing notation, we simply denote by $x$ the value $x(P_y)$ in the proof below. The integer $y$ being odd, we deduce that $y^n$ is odd, hence the polynomial $C_2P$ has integral coefficients. It follows that $x$ is a unit, hence defines an element in the Selmer group.

On the other hand, we have by construction
$$
(x-2)(2x-1) = -y^nx.
$$

We know that $y$ is coprime to $3$ and that $x$ is a unit, so the right-hand side is coprime to $3$. It follows that the two terms on the left-hand side are also coprime to $3$, hence are coprime because $(x-2,2x-1)$ divides $3$ by Lemma~\ref{lem:pgcd}. Therefore, $x-2$ and $2x-1$ are both $n$-th powers of ideals, according to the identity above.

\subsubsection*{Signature}

The discriminant of $C_2P$ is $(y^n-9)(y^n-1)/4$, which is strictly positive unless $1\leq y^n\leq 9$. So, if $y$ is large, the field $\Q(P_y)$ is totally real, with discriminant $O(y^{2n})$.

\subsubsection*{Statement of the result}

According to Theorem~\ref{thm:specialization}, for all but $O(\sqrt{N})$ values $y\in\{1,\dots, N\}$ satisfying the conditions above, $\Q(P_y)$ is a real quadratic field, and
$$
\rank_n \Cl(\Q(P_y))\geq 1.
$$

Moreover, for $N$ large enough, the number of isomorphism classes of such $\Q(P_y)$ when $y$ runs through $\{1,\dots, N\}$ is $\gg N/\log N$.
It follows that, for $X$ large enough, the number of such fields whose discriminant is bounded above by $X$ is $\gg X^{1/2n}/\log X$.

%
%
%
%


\subsection{$\Z/3\Z$}

We shall recover the result of Nakano \cite{Nakano86} on cyclic cubic fields with $n$-rank at least two. Our family is the same as his. In fact, it is constructed from the universal family of cyclic cubic fields given by Shanks \cite{Shanks1974}, known as the ``simplest cubic fields''.

The homography
$$
z\mapsto -1/(z+1)
$$
has order $3$. The orbits of $0$ and $1$ are respectively given by
$$
\orb(0)=\{0, -1, \infty\}; \quad \orb(1)=\{1, -1/2, -2\}.
$$

The corresponding polynomial is
\begin{equation*}
\label{eq:C3pol}
\begin{split}
C_3P &= \frac{1}{2}\left((x-1)(x+2)(2x+1) + y^nx(x+1)\right) \\
     &= x^3 + \left(\frac{y^n + 3}{2}\right)x^2 + \left(\frac{y^n - 3}{2}\right)x - 1.
\end{split}
\end{equation*}

Let $y\in \Z$, and let $P_y$ be the corresponding point on the curve defined by $C_3P$.

\subsubsection*{Congruence conditions}
\begin{enumerate}
\item[(i)] $2\nmid y$ and $3\nmid y$.
\end{enumerate}

One checks that, under these conditions, assumptions (A1) and (A2) are satisfied. It follows from the arguments in \S{}\ref{3.2.4} that the values at $P_y$ of the rational functions
\begin{equation*}
\label{eq:SelmerlistC3}
x, \quad x+1, \quad x-1, \quad x+2
\end{equation*}
define classes in $\Sel^n(\Q(P_y))$.

\subsubsection*{Signature}

The polynomial $C_3P$ has degree three, hence has at least one real root. Therefore, according to Remark~\ref{rmk:signature}, the field $\Q(P_y)$ is totally real, with discriminant $O(y^{4n})$.

\subsubsection*{Statement of the result}

According to Theorem~\ref{thm:specialization}, for all but $O(\sqrt{N})$ values $y\in\{1,\dots, N\}$ satisfying the conditions above, $\Q(P_y)$ is a cyclic cubic field, and
$$
\rank_n \Cl(\Q(P_y))\geq 2.
$$

Moreover, for $N$ large enough, the number of isomorphism classes of such $\Q(P_y)$ when $y$ runs through $\{1,\dots, N\}$ is $\gg N$.
It follows that, for $X$ large enough, the number of such fields whose discriminant is bounded above by $X$ is $\gg X^{1/4n}$.

%
%
%
%


\subsection{$\Z/4\Z$}

The results obtained in this section appear to be new. The family of cyclic quartic fields we consider is a subfamily of the one constructed by Gras \cite{Gras78}, also known as the ``simplest quartic fields'' by analogy with Shank's terminology.

The homography
$$
z\mapsto (z-1)/(z+1)
$$
has order $4$. The orbits of $0$ and $2$ are respectively given by
$$
\orb(0)=\{0, -1, \infty, 1\}; \quad \orb(2)=\{2, 1/3, -1/2, -3\}.
$$

The corresponding polynomial is
\begin{equation*}
\label{eq:C4pol}
\begin{split}
C_4P &= \frac{1}{6}\left((x-2)(2x+1)(x+3)(3x-1) + y^nx(x-1)(x+1)\right) \\
     &= x^4 + \left(\frac{y^n+7}{6}\right)x^3 - 6x^2 - \left(\frac{y^n+7}{6}\right)x + 1.
\end{split}
\end{equation*}

Let $y\in \Z$, and let $P_y$ be the corresponding point on the curve defined by $C_4P$.

\subsubsection*{Congruence conditions}
\begin{enumerate}
\item[(i)] $n$ is odd;
\item[(ii)] $y\equiv 5\pmod{12}$;
\item[(iii)] $5\nmid y$.   
\end{enumerate}

We claim that, under these conditions, the values at $P_y$ of the rational functions
\begin{equation*}
\label{eq:SelmerlistC4}
x, \quad \frac{x-1}{x+1}, \quad \frac{x+3}{x+1}, \quad x-2, \quad 2x+1, \quad \frac{(x+1)^2}{2},
\end{equation*}
define classes in $\Sel^n(\Q(P_y))$.

Condition (i) and (ii) imply that $y^n \equiv 5\pmod{12}$, in particular the polynomial $C_4P$ has integral coefficients, hence (A1) holds. Conditions (ii) and (iii) imply that $y$ is coprime to $2$ and $5$, hence (A2) holds. It follows from the arguments in \S{}\ref{3.2.4} that all the functions in the above list, except the last one, have the required property.

Let $q:=\frac{y^n-5}{12}$, which is an integer under our conditions. Then one computes, using Pari/GP, that the minimal polynomial of $(x(P_y)+1)^2/2$ is given by
$$
T^4 - 2(q^2 + q + 4)T^3 + (5q^2 + 10q + 19)T^2 - 2(q^2 +3q + 4)T + 1
$$
which has integral coefficients. Therefore $(x(P_y)+1)^2/2$ is a unit, hence the last function in the list satisfies the property.

\subsubsection*{Signature}

The polynomial $C_4P$ satisfies $C_4P(0)=1$ and $C_4P(1)=-4$, hence has at least one real root. Therefore, according to Remark~\ref{rmk:signature}, the field $\Q(P_y)$ is totally real, with discriminant $O(y^{6n})$.

\subsubsection*{Statement of the result}

Assume $n$ is odd. Then for all but $O(\sqrt{N})$ values $y\in\{1,\dots, N\}$ satisfying the conditions above, $\Q(P_y)$ is a real cyclic quartic field, and
$$
\rank_n \Cl(\Q(P_y))\geq 3.
$$

Moreover, for $N$ large enough, the number of isomorphism classes of such $\Q(P_y)$ when $y$ runs through $\{1,\dots, N\}$ is $\gg N$.
It follows that, for $X$ large enough, the number of such fields whose discriminant is bounded above by $X$ is $\gg X^{1/6n}$.

%
%
%
%


\subsection{$\Z/6\Z$}

The results obtained in this section appear to be new. Our family of cyclic sextic fields is similar to that constructed by Gras \cite{Gras86}, although the cyclic subgroup of $\PGL_2(\Q)$ on which her construction relies on is not the same than ours.

The homography
$$
z\mapsto (2z-1)/(z+1)
$$
has order $6$. The orbits of $0$ and $3$ are respectively given by
$$
\orb(0)=\{0, -1, \infty, 1, 2, 1/2 \}; \quad \orb(3)=\{3, 5/4, 2/3, 1/5, -1/2, -4\}.
$$

The corresponding polynomial is
\begin{equation*}
\label{eq:C6pol}
\begin{split}
C_6P &= \frac{1}{120}\left((x-3)(x+4)(2x+1)(3x-2)(4x-5)(5x-1) + y^nx(x-1)(x+1)(x-2)(2x-1)\right) \\
     &= x^6 + \left(\frac{y^n-37}{60}\right)x^5 - \left(\frac{y^n+323}{24}\right)x^4 + 20x^3 + \left(\frac{y^n-37}{60}\right)x^2 - \left(\frac{y^n+323}{24}\right)x + 1.
\end{split}
\end{equation*}

Let $y\in \Z$, and let $P_y$ be the corresponding point on the curve defined by $C_6P$.

\subsubsection*{Congruence conditions}
\begin{enumerate}
\item[(i)] $n$ is coprime to $6$;
\item[(ii)] $y\equiv 397\pmod{1080}$ if $n=12k+1$, $y\equiv 37 \pmod{1080}$ if $n=12k+5$, $y\equiv 613\pmod{1080}$ if $n=12k+7$ and $y\equiv 253\pmod{1080}$ if $n=12k+11$;
\item[(iii)] $7\nmid y$.  
\end{enumerate}

We claim that, under these conditions, the values at $P_y$ of the rational functions
\begin{equation*}
\label{eq:SelmerlistC6}
x, \quad x-1, \quad \frac{2x-1}{x+1}, \quad \frac{x-2}{x+1}, \\
\quad x-3, \quad \frac{x+4}{x+1}, \quad 2x+1, \quad 3x-2, \quad \frac{4x-5}{x+1}, \quad \frac{(x+1)^2}{3}
\end{equation*}
define classes in $\Sel^n(\Q(P_y))$.

Condition (ii) implies that $y^n\equiv 397\pmod{1080}$, in particular the polynomial $C_6P$ has integral coefficients, hence (A1) holds. Condition (ii) and (iii) imply that $y$ is coprime to $2$, $3$ and $7$, hence (A2) holds. It follows from the arguments in \S{}\ref{3.2.4} that all the functions in the above list, except the last one, have the required property.

Let $q:=\frac{y^n-397}{1080}$, which is an integer under our conditions. Then one computes, using Pari/GP, that the minimal polynomial of $(x(P_y)+1)^2/3$ is given by
\begin{multline*}
T^6 - 6(18q^2 + 15q + 5)T^5 + 15(39q^2 + 36q + 11)T^4 - 2(483q^2 + 483q + 151)T^3 \\ + 15(39q^2 + 42q + 14)T^2 - 6(18q^2 + 21q + 8)T + 1
\end{multline*}
which has integral coefficients. Therefore $(x(P_y)+1)^2/2$ is a unit, hence the last function in the list satisfies the property.

\subsubsection*{Signature}

The polynomial $C_6P$ satisfies $C_6P(0)=1$ and $C_6P(-1)=-27$, hence has at least one real root. Therefore, according to Remark~\ref{rmk:signature}, the field $\Q(P_y)$ is totally real, with discriminant $O(y^{10n})$.

\subsubsection*{Statement of the result}

Assume $n$ is coprime to $6$. Then for all but $O(\sqrt{N})$ values $y\in\{1,\dots, N\}$ satisfying the conditions above, $\Q(P_y)$ is a real cyclic sextic field, and
$$
\rank_n \Cl(\Q(P_y))\geq 5.
$$

Moreover, for $N$ large enough, the number of isomorphism classes of such $\Q(P_y)$ when $y$ runs through $\{1,\dots, N\}$ is $\gg N$.
It follows that, for $X$ large enough, the number of such fields whose discriminant is bounded above by $X$ is $\gg X^{1/10n}$.

%
%
%
%


\subsection{$D_2=(\Z/2\Z)^2$}

We shall construct, for any $n>1$, real biquadratic fields whose class group has $n$-rank at least three. A detailed analysis reveals that each of the three quadratic subfields has at least one element of order $n$ in its class group. While these examples appear to be new, it follows from Yamamoto's result that one can achieve $n$-rank at least $4$ in the imaginary case.

Let $G\simeq (\Z/2\Z)^2$ be the subgroup of $\PGL_2(\Q)$ generated by the homographies
$$
z \mapsto (z+1)/(z-1) \quad\text{and}\quad z \mapsto -1/z
$$
Choosing $a=2$, $b=0$ and $\lambda$ as in \S{}\ref{3.2.3}, we obtain from Lemma~\ref{lem:Picard} the polynomial
\begin{equation*}
\label{eq:D2pol}
\begin{split}
D_2P &= \frac{1}{6}\left((x-2)(2x+1)(x-3)(3x+1)+y^nx(x-1)(x+1)\right) \\
     &= x^4 + \left(\frac{y^n-25}{6}\right)x^3 + 2x^2 - \left(\frac{y^n-25}{6}\right)x + 1.
\end{split}
\end{equation*}

Let $y\in \Z$, and let $P_y$ be the corresponding point on the curve defined by $D_2P$.

\subsubsection*{Congruence conditions}
\begin{enumerate}
\item[(i)] $y\equiv 1 \pmod{12}$.
\item[(ii)] $5\nmid y$ and $7\nmid y$.  
\end{enumerate}

One checks that, under these conditions, the values at $P_y$ of the rational functions
\begin{equation*}
\label{eq:SelmerlistD2}
x, \quad \frac{x+1}{x-1}, \quad x-2, \quad 2x+1, \quad \frac{x-3}{x-1}, \quad \frac{(x-1)^2}{2},
\end{equation*}
define classes in $\Sel^n(\Q(P_y))$. In other terms, the condition \eqref{eq:SelCond} in Theorem~\ref{thm:specialization} is satisfied by these rational maps and this choice of $P_y$.

\subsubsection*{Signature}

The polynomial $D_2P$ satisfies $D_2P(0)=1$ and $D_2P(1/2)=(-y^n+50)/16$, hence has at least one real root when $y^n>50$. Therefore, the field $\Q(P_y)$ is totally real for $y$ large enough.

\subsubsection*{Statement of the result}

According to Theorem~\ref{thm:specialization}, for all but $O(\sqrt{N})$ values $y\in\{1,\dots, N\}$ satisfying the conditions above, $\Q(P_y)$ is a real biquadratic field, and
$$
\rank_n \Cl(\Q(P_y))\geq 3.
$$

%

\begin{remark}
Polynomials defining the three quadratic subfields are given by
\begin{equation*}
\begin{split}
x^2 + \left(\frac{y^n-25}{6}\right)x + 4 &= \frac{1}{6}\left((2x - 3)(3x - 8) + y^nx\right); \\
x^2 + \left(\frac{y^n-25}{6}\right)x + \left(\frac{y^n-25}{6}\right) &= \frac{1}{6}\left((x - 5)(6x + 5) + y^n(x+1)\right); \\
x^2 + \left(\frac{y^n-25}{6}\right)x - \left(\frac{y^n-25}{6}\right) &= \frac{1}{6}\left((2x - 5)(3x - 5) + y^n(x-1)\right);
\end{split}
\end{equation*}
and one checks that, under the above assumptions, each of these fields has at least one element of order $n$ in its class group. When $n$ is odd, this allows to recover the above result via Brauer's class number relations \cite{Walter79a}.
\end{remark}


%
%
%
%


\subsection{$D_3=\mathfrak{S}_3$}
\label{sec:D3}

The subgroup generated by
$$
z\mapsto -1/(z+1) \quad \text{and}\quad z\mapsto 1/z
$$
is isomorphic to $D_3=\mathfrak{S}_3$. The orbits of $0$ and $2$ are respectively given by
$$
\orb(0)=\{0, -1, \infty \}; \quad \orb(2)=\{2, -1/3, -3/2, 1/2, -3, -2/3\}.
$$

The corresponding polynomial is
\begin{equation*}
\label{eq:D3pol}
\begin{split}
D_3P &= \frac{1}{36}\left((x-2)(2x-1)(x+3)(3x+1)(2x+3)(3x+2) + y^nx^2(x+1)^2\right) \\
     &= x^6 + 3x^5 + \left(\frac{y^n-127}{36}\right)x^4 + \left(\frac{y^n-217}{18}\right)x^3 + \left(\frac{y^n-127}{36}\right)x^2 + 3x + 1.
\end{split}
\end{equation*}

Let $y\in \Z$, and let $P_y$ be the corresponding point on the curve defined by $D_3P$.

\subsubsection*{Congruence conditions}
\begin{enumerate}
\item[(i)] $n$ is odd;
\item[(ii)] $y\equiv 19\pmod{36}$;
\item[(iii)] $5\nmid y$ and $7\nmid y$.  
\end{enumerate}

Under these conditions, the values at $P_y$ of the rational functions
\begin{equation*}
\label{eq:SelmerlistD3}
x, \quad x+1, \quad x-2, \quad x+3, \quad 2x-1, \quad 2x+3, \quad 3x+1
\end{equation*}
define classes in $\Sel^n(\Q(P_y))$.

\subsubsection*{Signature}

The polynomial $D_3P$ can be written as
$$
D_3P = P_1 + y^n P_2.
$$

We note that $P_1$ is monic of degree $6$, hence has a global minimum over $\mathbb{R}$; that $P_2$ takes strictly positive values over $\mathbb{R}$, except at its roots $0$ and $-1$; and that $P_1(0)$ and $P_1(-1)$ are strictly positive. By elementary calculus it follows from these observations that, for $y^n$ large enough, $D_3P$ takes strictly positive values over $\mathbb{R}$, hence the field $\Q(P_y)$ is totally imaginary.
A detailed analysis reveals that, for $y^n>100$, $D_3P$ takes strictly positive values.

\subsubsection*{Statement of the result}

Assume $n$ is odd. Then for all but $O(\sqrt{N})$ values $y\in\{1,\dots, N\}$ satisfying the conditions above, $\Q(P_y)$ is a totally imaginary Galois extension of $\Q$ with group $\mathfrak{S}_3$, and
$$
\rank_n \Cl(\Q(P_y))\geq 5.
$$

\begin{remark}
The quadratic subfield is defined by the polynomial
$$
x^2 + 3x + \left(\frac{y^n-19}{36}\right) = \frac{1}{36}\left((6x-1)(6x+19) + y^n\right).
$$
The three conjugate cubic subfields are defined by
$$
x^3 + 3x^2 + \left(\frac{y^n-235}{36}\right)x + \left(\frac{y^n-325}{18}\right) = \frac{1}{36}\left((2x-5)(3x+10)(6x+13) + y^n(x+2)\right).
$$

One checks that, under the above assumptions, the class group of the quadratic (resp. cubic) subfield has $n$-rank at least one (resp. two).  When $n$ is coprime to $6$, this allows to recover the above result via Brauer's class number relations \cite{Walter79a}.
\end{remark}

\begin{remark}
The above factorizations are particular instances of a general phenomenon: let $f(z)=-1/(z+1)$, and let $a\in\Q$ which does not belong to the orbit of $0$, then the splitting field of the polynomial
$$
\prod_{i=1}^3 \left(x-f^i(a)\right)\left(x-\frac{1}{f^i(a)}\right) - tx^2(x+1)^2
$$
is Galois over $\Q(t)$ with group $\mathfrak{S}_3$, according to our construction. Its quadratic subfield is defined by the polynomial
$$
\left(x-(a+f(a)+f^2(a))\right)\left(x-\left(\frac{1}{a}+\frac{1}{f(a)}+\frac{1}{f^2(a)}\right)\right) - t,
$$
and the cubic subfields are defined by
$$
\prod_{i=1}^3 \left(x-\left(f^i(a)+\frac{1}{f^i(a)}\right)\right) - t(x+2).
$$

Needless to say, this can be generalized to arbitrary values of $a$, $b$ and $r$ (the order of $f$).
\end{remark}

\begin{example}
Let $n=5$ and $y=199$, then our field is defined by the polynomial
$$
x^6 + 3x^5 + 8668877802x^4 + 17337755599x^3 + 8668877802x^2 + 3x + 1
$$

One checks using Pari/GP that its quadratic subfield has class group of $5$-rank $2$, and that its cubic subfields have $5$-rank $2$ each. It follows from Brauer's class number relations that the class group of this field has $5$-rank $6$. In view of our numerical experiments, it seems that there should exist infinitely many such cases in our family.
\end{example}

%
%
%
%


\subsection{$D_4$}

The subgroup generated by
$$
z\mapsto (z-1)/(z+1) \quad \text{and}\quad z\mapsto 1/z
$$
is isomorphic to $D_4$. The orbits of $0$ and $2$ are respectively given by
$$
\orb(0)=\{0, -1, \infty, 1 \}; \quad \orb(2)=\{2, 1/3, -1/2, -3, 1/2, 3, -2, -1/3\}.
$$

The corresponding polynomial is
\begin{equation*}
\label{eq:D4pol}
\begin{split}
D_4P &= \frac{1}{36}\left((x-2)(2x-1)(x+2)(2x+1)(x-3)(3x-1)(x+3)(3x+1) + y^nx^2(x-1)^2(x+1)^2\right) \\
     &= x^8 + \left(\frac{y^n - 481}{36}\right)x^6 - \left(\frac{y^n - 733}{18}\right)x^4 + \left(\frac{y^n - 481}{36}\right)x^2 + 1.
\end{split}
\end{equation*}

Let $y\in \Z$, and let $P_y$ be the corresponding point on the curve defined by $D_4P$.

\subsubsection*{Congruence conditions}
\begin{enumerate}
\item[(i)] $n$ is coprime to $6$;
\item[(ii)] $y\equiv 49\pmod{144}$ if $n=3k+1$ and $y\equiv 97\pmod{144}$ if $n=3k+2$;
\item[(iii)] $5\nmid y$ and $7\nmid y$.  
\end{enumerate}

Under these conditions, the values at $P_y$ of the rational functions
\begin{equation*}
\label{eq:SelmerlistD4}
x, \quad \frac{x-1}{x+1}, \quad x-2, \quad x+2, \quad \frac{x-3}{x+1}, \quad \frac{x+3}{x+1}, \quad 2x-1, \quad 2x+1, \quad \frac{3x-1}{x+1}, \quad \frac{(x+1)^2}{2}
\end{equation*}
define classes in $\Sel^n(\Q(P_y))$.

\subsubsection*{Signature}

A detailed analysis, similar to the $D_3P$ case, reveals that, when $y^n>49$, the polynomial $D_4P$ takes strictly positive values over $\mathbb{R}$, hence the field $K$ is totally complex.

\subsubsection*{Statement of the result}

Assume $n$ is coprime to $6$. Then for all but $O(\sqrt{N})$ values $y\in\{1,\dots, N\}$ satisfying the conditions above, $\Q(P_y)$ is a totally complex Galois extension of $\Q$ with group $D_4$, and
$$
\rank_n \Cl(\Q(P_y))\geq 7.
$$

\begin{remark}
Replacing $y^n$ by $y^{2n}$ in the polynomial $D_4P$, one can prove similar result when $n$ is only assumed to be coprime to $3$. It suffices to  generalize our constructions, starting from Lemma~\ref{lem:Picard}, in order to take into account the square factors in the orbit at infinity.
\end{remark}

\begin{remark}
Choosing $a=4$ instead, one finds another congruence condition on $y^n$, which has a solution for all $n$ coprime to $5$. Therefore, the same result hold for $n$ coprime to $10$ (in fact, $5$ if one takes into account the previous remark).
\end{remark}

%
%
%
%


\subsection{$D_6$}

The subgroup generated by
$$
z\mapsto (2z-1)/(z+1) \quad \text{and}\quad z\mapsto 1/z
$$
is isomorphic to $D_6$. The orbits of $0$ and $-2$ are respectively given by
$$
\orb(0)=\{0, -1, \infty, 2, 1, 1/2\}; \quad \orb(-2)=\{-2, 5, 3/2, 4/5, 1/3, -1/4, \text{and their inverses}\}.
$$

The corresponding polynomial is
\begin{multline*}
\label{eq:D6pol}
D_6P = \frac{1}{120^2}((x+2)(2x+1)(x-3)(3x-1)(x+4)(4x+1)(x-5)(5x-1)(2x-3)(3x-2)(4x-5)(5x-4) \\ + y^nx^2(x-1)^2(x+1)^2(x-2)^2(2x-1)^2)
\end{multline*}

Let $y\in \Z$, and let $P_y$ be the corresponding point on the curve defined by $D_6P$.

\subsubsection*{Congruence conditions}
\begin{enumerate}
\item[(i)] $n$ is coprime to $6$;
\item[(ii)] $y^n\equiv 117649 \pmod{388800}$;
\item[(iii)] $7\nmid y$, $11\nmid y$ and $13\nmid y$. 
\end{enumerate}

Under these conditions, the values at $P_y$ of the rational functions
\begin{multline*}
x, \quad x-1, \quad \frac{2x-1}{x+1}, \quad \frac{x-2}{x+1}, \quad x+2, \quad \frac{x-5}{x+1}, \quad 2x-3, \quad \frac{5x-4}{x+1}, \quad 3x-1, \quad \frac{4x+1}{x+1},\\
2x+1, \quad \frac{5x-1}{x+1}, \quad 3x-2, \quad \frac{4x-5}{x+1}, \quad x-3, \quad \frac{(x+1)^2}{3}
\end{multline*}
define classes in $\Sel^n(\Q(P_y))$.

\subsubsection*{Signature}

As before, a detailed analysis reveals that, when $y^n>20449$, the polynomial $D_6P$ takes strictly positive values over $\mathbb{R}$, hence the field $K$ is totally complex.

\subsubsection*{Statement of the result}

Assume $n$ is coprime to $6$. Then for all but $O(\sqrt{N})$ values $y\in\{1,\dots, N\}$ satisfying the conditions above, $\Q(P_y)$ is a totally complex Galois extension of $\Q$ with group $D_6$, and
$$
\rank_n \Cl(\Q(P_y))\geq 11.
$$

%
%
%
%


\bibliographystyle{amsalpha}
\bibliography{PGL2biblio}

\providecommand{\bysame}{\leavevmode\hbox to3em{\hrulefill}\thinspace}
\providecommand{\MR}{\relax\ifhmode\unskip\space\fi MR }
\providecommand{\MRhref}[2]{%
  \href{http://www.ams.org/mathscinet-getitem?mr=#1}{#2}
}
\providecommand{\href}[2]{#2}
\begin{thebibliography}{{The}19}

\bibitem[AC55]{AC55}
N.~C. Ankeny and S.~Chowla, \emph{On the divisibility of the class number of
  quadratic fields}, Pacific J. Math. \textbf{5} (1955), 321--324. \MR{85301}

\bibitem[Bea10]{Beauville2010}
Arnaud Beauville, \emph{Finite subgroups of {${\rm PGL}_2(K)$}}, Vector bundles
  and complex geometry, Contemp. Math., vol. 522, Amer. Math. Soc., Providence,
  RI, 2010, pp.~23--29. \MR{2681719}

\bibitem[BG18]{bg18}
Yuri Bilu and Jean Gillibert, \emph{Chevalley-{W}eil theorem and subgroups of
  class groups}, Israel J. Math. \textbf{226} (2018), no.~2, 927--956.
  \MR{3819714}

\bibitem[BL17]{bl17}
Yuri Bilu and Florian Luca, \emph{Diversity in parametric families of number
  fields}, Number theory---{D}iophantine problems, uniform distribution and
  applications, Springer, Cham, 2017, pp.~169--191. \MR{3676399}

\bibitem[Bru65]{brumer65}
A.~Brumer, \emph{Ramification and class towers of number fields}, Michigan
  Math. J. \textbf{12} (1965), 129--131. \MR{0202697}

\bibitem[Coh81]{cohen1981}
S.~D. Cohen, \emph{The distribution of {G}alois groups and {H}ilbert's
  irreducibility theorem}, Proc. London Math. Soc. (3) \textbf{43} (1981),
  no.~2, 227--250. \MR{628276}

\bibitem[CZ03]{cz03}
Pietro Corvaja and Umberto Zannier, \emph{On the number of integral points on
  algebraic curves}, J. Reine Angew. Math. \textbf{565} (2003), 27--42.
  \MR{2024644}

\bibitem[DZ94]{DZ94}
R.~Dvornicich and U.~Zannier, \emph{Fields containing values of algebraic
  functions}, Ann. Scuola Norm. Sup. Pisa Cl. Sci. (4) \textbf{21} (1994),
  no.~3, 421--443. \MR{1310635}

\bibitem[GG19]{gg19}
Jean Gillibert and Pierre Gillibert, \emph{On the splitting of the {K}ummer
  exact sequence}, Publications Math\'ematiques de Besan\c con (2019), no.~2,
  19--27 (en).

\bibitem[GL12]{gl12}
Jean Gillibert and Aaron Levin, \emph{Pulling back torsion line bundles to
  ideal classes}, Math. Res. Lett. \textbf{19} (2012), no.~6, 1171--1184.
  \MR{3091601}

\bibitem[GL20]{gl20}
\bysame, \emph{A geometric approach to large class groups: a survey}, Class
  Groups of Number Fields and Related Topics (Kalyan Chakraborty, Azizul Hoque,
  and Prem~Prakash Pandey, eds.), Springer, 2020, pp.~1--15.

\bibitem[Gra78]{Gras78}
Marie-Nicole Gras, \emph{Table num\'erique du nombre de classes et des unit\'es
  des extensions cycliques r\'eelles de degr\'e~4 de $\protect \mathbb{Q}$},
  Publications Math\'ematiques de Besan\c con - Alg\`ebre et Th\'eorie des
  Nombres (1978), no.~2 (fr).

\bibitem[Gra86]{Gras86}
\bysame, \emph{Familles d'unit\'es dans les extensions cycliques r\'eelles de
  degr\'e~6 de $\protect \mathbb{Q}$}, Publications Math\'ematiques de Besan\c
  con - Alg\`ebre et Th\'eorie des Nombres (1986), no.~2 (fr).

\bibitem[Lev07]{levin07}
Aaron Levin, \emph{Ideal class groups, {H}ilbert's irreducibility theorem, and
  integral points of bounded degree on curves}, J. Th\'{e}or. Nombres Bordeaux
  \textbf{19} (2007), no.~2, 485--499. \MR{2394898}

\bibitem[Nak84]{Nakano84}
Shin Nakano, \emph{On ideal class groups of algebraic number fields}, Proc.
  Japan Acad. Ser. A Math. Sci. \textbf{60} (1984), no.~2, 74--77. \MR{750283}

\bibitem[Nak85]{Nakano85}
\bysame, \emph{On ideal class groups of algebraic number fields}, J. Reine
  Angew. Math. \textbf{358} (1985), 61--75. \MR{797674}

\bibitem[Nak86]{Nakano86}
\bysame, \emph{Ideal class groups of cubic cyclic fields}, Acta Arith.
  \textbf{46} (1986), no.~3, 297--300. \MR{864264}

\bibitem[RZ69]{RZ}
P.~Roquette and H.~Zassenhaus, \emph{A class of rank estimate for algebraic
  number fields}, J. London Math. Soc. \textbf{44} (1969), 31--38. \MR{0236146}

\bibitem[Ser92]{Serretopics}
Jean-Pierre Serre, \emph{Topics in {G}alois theory}, Research Notes in
  Mathematics, vol.~1, Jones and Bartlett Publishers, Boston, MA, 1992, Lecture
  notes prepared by Henri Damon [Henri Darmon], With a foreword by Darmon and
  the author. \MR{1162313}

\bibitem[Sha74]{Shanks1974}
Daniel Shanks, \emph{The simplest cubic fields}, Math. Comp. \textbf{28}
  (1974), 1137--1152. \MR{352049}

\bibitem[{The}19]{PARI2}
{The PARI~Group}, Univ. Bordeaux, \emph{{PARI/GP version \texttt{2.11.2}}},
  2019, available from \texttt{http://pari.math.u-bordeaux.fr/}.

\bibitem[Wal79]{Walter79a}
C.~D. Walter, \emph{Brauer's class number relation}, Acta Arith. \textbf{35}
  (1979), no.~1, 33--40. \MR{536878}

\bibitem[Wei73]{Weinberger73}
P.~J. Weinberger, \emph{Real quadratic fields with class numbers divisible by
  {$n$}}, J. Number Theory \textbf{5} (1973), 237--241. \MR{335471}

\bibitem[Yam70]{Yamamoto70}
Yoshihiko Yamamoto, \emph{On unramified {G}alois extensions of quadratic number
  fields}, Osaka Math. J. \textbf{7} (1970), 57--76. \MR{266898}

\end{thebibliography}



\bigskip

\textsc{Jean Gillibert}, Institut de Math{\'e}matiques de Toulouse, CNRS UMR 5219, 118 route de Narbonne, 31062 Toulouse Cedex 9, France.

\emph{E-mail address:} \texttt{jean.gillibert@math.univ-toulouse.fr}
\medskip

\textsc{Pierre Gillibert}, La Vieille \'Eglise, 14370 Vimont, France

\emph{Email address:} \texttt{pgillibert@yahoo.fr}


\end{document}